\documentclass[preprint,review,12pt]{amsart}

\usepackage{amsfonts,amssymb,amsmath,amsthm}
\usepackage[cp1251]{inputenc}
\usepackage[T2B]{fontenc}
\usepackage[english]{babel}
\usepackage[mathscr]{eucal}
\usepackage{calrsfs}

\newtheorem{theorem}{Theorem}[section]
\newtheorem{lemma}[theorem]{Lemma}

\theoremstyle{definition}

\newtheorem{example}[theorem]{Example}

\theoremstyle{remark}

\numberwithin{equation}{section}

\begin{document}
\title[Schatten norm estimates of Hardy-Steklov operator]{Estimates for Schatten-von Neumann norms of Hardy-Steklov operator}
\author{Elena P. Ushakova}
\curraddr{Department of Mathematics, University of York, York, YO10 5DD, United Kingdom.}\email{elena.ushakova@york.ac.uk}

\address{Computing Centre of the Far-Eastern Branch of the Russian Academy of Sciences, Khabarovsk, 680000, RUSSIA.}
\email{elenau@inbox.ru}

\begin{abstract} Some upper and lower estimates are obtained for Schatten-von Neumann norms of Hardy-Steklov operator in Lebesgue spaces on the semi-axis.
\end{abstract}
\keywords{Hardy-Steklov operator, Lebesgue space, Approximation number, Schatten class, Schatten-von Neumann norm}
\subjclass{47G10}
\maketitle
\section{Introduction}
Given $1<p<\infty$ and $I\subseteq[0,+\infty)=:\mathbb{R}^+$ let $L^p(I)$ denote a collection of measurable functions $f$ on $I$ satisfying $\|f\|_{p,I}:=\bigl(\int_I |f(t)|^p\mathrm{d}t\bigr)^{\frac{1}{p}}<\infty.$
Put $p'=p/(p-1)$ and assume that $v,w$ are non-negative weight functions on $\mathbb{R}^+$ such that $v\in L^{p'}_{\rm loc}(0,\infty)$ and $w\in L^p_{\rm loc}(0,\infty).$ We consider the Hardy -Steklov operator \begin{equation}\label{utro}\mathcal{H}f(x)=w(x)\int_{a(x)}^{b(x)}f(y)v(y)\mathrm{d}y\end{equation}
on $L^p(\mathbb{R}^+)$ with boundaries $a(x)$ and $b(x)$ satisfying the conditions:
\begin {equation*}
\label {3}
\begin {tabular}{ll} (i) &
$a(x),$ $b(x)$ are differentiable and strictly increasing on
$(0,\infty);$\\ (ii) & $a(0)=b(0)=0,$ $a(x)< b(x)$ for
$0<x<\infty,$ $a(\infty)=b(\infty)=\infty.$
\end {tabular}
\end {equation*}
The operator $\mathcal{H}$ is a generalization of a weighted Hardy integral operator of the form $Hf(x)=w(x)\int_0^xf(t)v(t)\mathrm{d}t,$ which mapping properties on $L^p(\mathbb{R}^+)$ were effectively studied in a number of articles and books (see e.g. \cite{KP2003,EEH1988,EEH1997,LL1999} and references there). Having more complicated structure than $H$ the Hardy-Steklov transform \eqref{utro} is significantly more difficult to study. Nevertheless, there exist several results concerning boundedness and compactness properties of $\mathcal{H}$ in Lebesgue spaces on $\mathbb{R}^+$ (see \cite{HS1998} and \cite{SU2001,SU,SUB,JFSA}). In this article we deal with another question, which is related to Schatten ideal behaviour of $\mathcal{H}:L^p(\mathbb{R}^+)\to L^p(\mathbb{R}^+)$. Today, this problem is very poorly understood. Our result is necessary and sufficient conditions for belonging $\mathcal{H}:L^p(\mathbb{R}^+)\to L^p(\mathbb{R}^+)$ to the Schatten-von Neumann class $\mathbf{S}_\alpha.$ Remind that all compact on $X$ linear operators $T:X\to X$ satisfying
$$\|T\|_{\mathbf{S}_\alpha}:=\biggl(\sum_{n\in\mathbb{N}}a_n^\alpha(T)\biggr) ^{\frac{1}{\alpha}}<\infty,\hspace{0.5cm}0<\alpha<\infty,$$ constitute the Schatten classes $\mathbf{S}_\alpha.$  Here $a_n(T)$ is the $n$-th approximation number of the operator $T:X\to X$ defined
$$a_n(T)=\inf_{K\colon\,{\rm rank} K<n}\|T-K\|_{X\to X},\hspace{5mm}n=1,2,\ldots.$$ The quantity $\|\cdot\|_{\mathbf{S}_\alpha}$ is called the Schatten-von Neumann norm. Our necessary and sufficient conditions obtained have a form of a two-sided estimate of $\|\mathcal{H}\|_{\mathbf{S}_\alpha}$ by functionals expressed in terms of $p,$ $v,$ $w$ and $\alpha.$

The pointed problem is not enough studied today, in particular, for the case $p\not=2.$ The only known result is by E.N. Lomakina \cite{LomII}, where the author is giving  a criterion for $\mathcal{H}:L^2(\mathbb{R}^+)\to L^2(\mathbb{R}^+)$ to be in $\mathbf{S}_\alpha$ \cite[Theorem 5]{LomII} and a sufficient condition for belonging $\mathcal{H}:L^p(\mathbb{R}^+)\to L^p(\mathbb{R}^+)$ with $1<p<\infty$ to the class  $\mathbf{S}_\alpha$ for all $\alpha>1$ \cite[(6)]{LomII}. The sufficient condition by E.N. Lomakina was found with help of known upper estimates for the norm $\|H\|_{\mathbf{S}_\alpha}$ of the Hardy operator $H$ \cite{EEH1988,EEH1997,LomII} and on a base of a one-sided inequality binding counting functions of the sequences $a_n(\mathcal{H})$ and $a_n(H)$ \cite[Lemma 2]{LomI}. That result has a form of an upper estimate of $\|H\|_{\mathbf{S}_\alpha}$ by a  discrete functional. In Theorem \ref{lom} of our article we find upper bounds for that functional in a continuous form (see \eqref{p<1} and \eqref{p>1}).

Our main result is a lower estimate for the norm $\|\mathcal{H}:L^p(\mathbb{R}^+)\to L^p(\mathbb{R}^+)\|_{\mathbf{S}_\alpha},$ when $\alpha>0$  (see Theorem \ref{main'}). We also give an alternative upper estimate for $\|\mathcal{H}\|_{\mathbf{S}_\alpha}$ (see Theorem \ref{notmain}), which is different from that by E.N. Lomakina but convertible to the forms analogous to those in Theorem \ref{lom}. In order to obtain our results we directly applied the original method by D.E. Edmunds, W.D. Evans, D.J. Harris (see \cite{EEH1988}, \cite{EEH1997} and \cite{LS2000}), which was initially created for $H$, to the Hardy-Steklov operator of the form \eqref{utro}. Such a way allowed us to obtain a two-sided estimate for $\|\mathcal{H}\|_{\mathbf{S}_\alpha}$ with upper and lower bounds in discrete forms (see Theorems \ref{notmain} and \ref{main'}). Compared to the lower bound the one from above was convertible to an integral form \eqref{int} (or \eqref{int'}). Meanwhile, the estimate from below seemed too "small" for that purpose. That is why our necessary and sufficient conditions are of different types and provide a criterion for certain boundaries $a,b$ and weights $v,w$ only (see Example \ref{ex}). Nevertheless, being the only existing necessary condition for $\mathcal{H}:L^p(\mathbb{R}^+)\to L^p(\mathbb{R}^+),$ $p\not=2,$ to belong to the class $\mathbf{S}_\alpha,$ $\alpha>0,$ the lower estimate in Theorem \ref{main'} is the most valuable of our findings in this work. There might be a possibility to improve the result obtained to a  criterion with the same necessary and sufficient conditions of the forms \eqref{a<p} or \eqref{a>p}. This option is totally depending on outcome of Lemma \ref{lem1}. Unfortunately, in our work this statement has a "viewless zone" or a kind of gap in its necessary part. Filling the gap would help to make the lower bound in the two-sided estimate for $\|\mathcal{H}\|_{\mathbf{S}_\alpha}$ of the same type as its upper one and, therefore, convertible to the form \eqref{a<p} (or \eqref{a>p}) for all $a,b,v$ and $w$. For the moment such an improvement seems impossible.

Throughout  the article products of the form $0\cdot\infty$ are supposed to be equal to 0. We write $A\ll B$ or $A\gg B$ when $A\le c_1 B$ or $A\ge c_2 B$ with constants $c_i,$ $i=1,2,$ which are either absolute or depending on $p$ or $\alpha$ only. $A\approx B$ means $A\ll B\ll A.$ Symbols $\mathbb{Z}$ and $\mathbb{N}$ denote integers $\{k\}$ and naturals $\{n\}$ respectively. $\chi_E$ stands for a characteristic function of a subset $E\subset\mathbb{R}^+.$ We also use $=:$ and $:=$ for marking new quantities.

\section{The main result}
\subsection{Preliminary estimates}
Let $I=(d,e)$ and $W_I:=\int_Iw^p(x)\mathrm{d}x.$ We shall write $$H(x)=\int_{a(x)}^{b(x)}f(y)v(y)\mathrm{d}y,\ \ \ H_I=\frac{1}{W_I}\int_IH(x)w^p(x)\mathrm{d}x.$$ Denote $$\mathcal{K}(I):=\sup_{f\not=0}\frac{\|w(H-H_I)\|_{p,I}}{\|f\|_{p}}.$$

Given $I\subset\mathbb{R}^+$ let $c\in I=[d,e]$ be chosen so that $\int_d^{c}w^p=\frac{1}{2}\int_d^ew^p.$ Denote $w_d=w\chi_{[d,c]},$  $w_e=w\chi_{[c,e]},$ $f_a=v\chi_{[a(d),a(c)]},$ $f_b=v\chi_{[b(c),b(e)]}$ and $${\bar H}(x)=\int_{a(d)}^{a(x)}f(y)v(y)\mathrm{d}y+ \int_{b(x)}^{b(e)}f(y)v(y)\mathrm{d}y.$$
The following statement is giving a two-sided estimate for $\mathcal{K}(I).$
\begin{lemma}\label{lem1}
We have \begin{eqnarray}\label{Ktwo}
\frac{1}{4}\biggl[\,\sup_{f\colon {\rm supp}\,f\subseteq[a(d),a(c)]}\frac{\|w_dH\|_{p}}{\|f\|_p}+
\sup_{f\colon {\rm supp}\,f\subseteq[b(c),b(d)]}\frac{\|w_eH\|_{p}}{\|f\|_p}\biggr]\nonumber\\
\le \mathcal{K}(I)\le 2\sup_{f\in L^p}\frac{\|w{\bar H}\|_{p,I}}{\|f\|_p}.\end{eqnarray}
\end{lemma}\begin{proof} To prove the lower estimate we take
$f=f_a$ and write
\begin{eqnarray*}\label{P11}\mathcal{K}(I)\|f_a\|_{p}\ge\|w_d(H-H_I)\|_{p}
\ge\|w_dH\|_{p}-\|w_dH_I\|_{p}\nonumber\\=\|w_dH\|_{p}-|H_I|W_{[d,c]}^{\frac{1}{p}}.
\end{eqnarray*} In view of $\int_d^cw^p=\frac{1}{2}\int_d^ew^p$ we have
\begin{eqnarray}\label{P12}\mathcal{K}(I)\|f_a\|_{p}\ge\|w_dH\|_{p}-\frac{\int_d^c H w^p}{2^{1/p}W_{[d,c]}^{\frac{1}{p'}}},
\end{eqnarray} and, by H\"{o}lder's inequality,
\begin{eqnarray}\label{P10}\mathcal{K}(I)\|f_a\|_{p}\ge\frac{1}{2}\|w_dH\|_{p}.\end{eqnarray}
Analogously, with $f=f_b$ we can obtain the estimate
\begin{eqnarray}\label{P10'}\mathcal{K}(I)\|f\|_{p,[b(c),b(d)]}\ge\frac{1}{2}\|w_eH\|_{p},\end{eqnarray}
which yields the required lower estimate when combined with \eqref{P10}.

The upper estimate is following from the fact that
$H(x)-H_I=-[{\bar H}(x)-{\bar H}_I]$ and the inequality
$$\|w({\bar H}-{\bar H}_I)\|_{p,I}\le 2 \|w{\bar H}\|_{p,I}.$$
\end{proof}

Throughout this section we shall broadly use the fact that $\mathcal{K}(I)=\mathcal{K}(d,e)$ is continuously depending on an interval $I=(d,e)$. This follows from the above statement or from the equivalence
$$\|w(H-H_I)\|_{p,I}\approx W_I^{-\frac{1}{p}}\biggl(\int_I\int_I\bigl|H(t)-H(s)\bigr|^pw^p(s)w^p(t) \mathrm{d}s\mathrm{d}t\biggr)^{\frac{1}{p}}$$ and arguments similar to those in \cite[Lemma 5]{EEH1988}. Therefore, given $0<M\varepsilon<\mathcal{K}(\mathbb{R}^+),$ $M\in\mathbb{N},$ there exists $N\in\mathbb{N}$ and numbers $0=c_0<c_1\ldots c_{MN}<c_{MN+1}=\infty,$ $M\in\mathbb{N},$ such that $\mathcal{K}(I_n)=\varepsilon$ for $n=0,\ldots,MN-1,$ where $I_n=(c_n,c_{n+1}),$ and $\mathcal{K}(I_{MN})\le\varepsilon.$ With all these assumptions we claim the following statement.
\begin{lemma}\label{Key}
Let $1<p<\infty$ and $0<7\varepsilon<\mathcal{K}(\mathbb{R}^+).$ Suppose that there exists $N\in\mathbb{N}$ and numbers $0=c_0<c_1\ldots c_{7N}<c_{7N+1}=\infty$ such that $\mathcal{K}(I_n)=\varepsilon$ for $n=0,\ldots,7N-1,$ where $I_n=(c_n,c_{n+1}),$ and $\mathcal{K}(I_{7N})\le\varepsilon.$ Then $a_N(\mathcal{H})\ge \varepsilon/2.$
\end{lemma}
\begin{proof} Let $\lambda\in(0,1).$ By assumption that $\mathcal{K}(I_n)=\varepsilon$ for $n=0,\ldots,7N-1$ there exist $f_n$ such that \begin{equation}\label{44}\|w(F_n-(F_n)_{I_n})\|_{q,I_n}> \lambda\varepsilon\|f_n\|_p,\end{equation}
where $F_n(x)=\int_{a(x)}^{b(x)}f_n(y)v(y)dy$ and the intervals $I_n$ are of two types. The first type, say $\mathcal{I}_1,$ is consisting of all $I_n,$ $n\in\mathbb{N}_1\subseteq\{1,\ldots,7N\},$ with end points $c_n$ and $c_{n+1}$ satisfying the property $b(c_n)\le a(c_{n+1}).$ The second type $\mathcal{I}_2$ is formed from all the rest $I_n$, $n\in\mathbb{N}_2\subseteq\{1,\ldots,7N\},$ $\mathbb{N}_1\cap\mathbb{N}_2=\emptyset,$ that is from those satisfying $b(c_n)> a(c_{n+1}).$

Define a dominant class $\mathcal{I}_d$ of intervals $I_n$ as follows: $$\mathcal{I}_d=\begin{cases}\mathcal{I}_1, & \mathrm{if}\ \ \ 
\sharp\{I_n\in\mathcal{I}_1\}> 2N,\\
\mathcal{I}_2, & \mathrm{if}\ \ \ \sharp\{I_n\in\mathcal{I}_2\}> 5N,\\
\textrm{either}\,\mathcal{I}_1\,\textrm{or}\,\mathcal{I}_2, & \mathrm{if}\ \ \ 
\sharp\{I_n\in\mathcal{I}_1\}=2N\,\textrm{and}\,\sharp\{I_n\in\mathcal{I}_2\}=5N.\end{cases}$$

We shall consider such a dominant class $\mathcal{I}_d$ of intervals $I_n$ and will take into account $N$ its members only, say, $I_m\in\mathcal{I}_d,$ $m=1,\ldots,N.$ Given $N$ the set of $I_m$ will be chosen depending on $\mathcal{I}_d.$

Further, chosen $\{I_m\}_{m=1}^N$ we introduce functions $$\phi_m=f_m\chi_{[a(c_m),a(c_{m+1})]}+f_m\chi_{[b(c_m),b(c_{m+1})]},\ \ \ m=1,\ldots,N,$$ and suppose that a linear operator $K$ is of rank less than $N.$ Then there exist constants $\lambda_1,\ldots,\lambda_N,$ not all zero, such that
\begin{equation}\label{K1}K\biggl(\sum_{m=1}^N\lambda_m \phi_m\biggr)=0.\end{equation} Besides, we define $\phi=\sum_{m=1}^N\lambda_m \phi_m$ and put $\Phi(x)=\int_{a(x)}^{b(x)}\phi(y)v(y)\mathrm{d}y.$

Consider first the situation when $\mathcal{I}_d=\mathcal{I}_1.$ Notice that in view of $b(c_n)\le a(c_{n+1})$ we have $\phi_n=f_n\chi_{[a(c_n),b(c_{n+1})]}.$ Moreover, the supports of $\phi_n$ with only odd serial numbers $n\in\mathbb{N}_1$ (or only even serial numbers $n\in\mathbb{N}_1$) are disjoint. Let us take as $I_m\in\mathcal{I}_d=\mathcal{I}_1$ exactly $N$ intervals $I_n\in\mathcal{I}_1$ with even serial numbers $n\in\mathbb{N}_1$. Notice that in such a case  \begin{equation*}\Phi(x)=\int_{a(x)}^{b(x)}\phi(y)v(y)\mathrm{d}y =\lambda_mF_m(x),\hspace{1cm}x\in I_m.\end{equation*}
It is known (see e.g. \cite[p. 482]{EEH1988}) that for all constants $k,$ \begin{eqnarray}\label{mark}\|w(F-F_I)\|_{p,I}\le \|w(F-k)\|_{p,I}+\|w(k-F)_I\|_{p,I}\nonumber\\
\le 2\|w(F-k)\|_{p,I}.\end{eqnarray}
Thus, by \eqref{K1}, \eqref{mark} and \eqref{44}
\begin{eqnarray*}\|\mathcal{H}\phi-K\phi\|_{p,\mathbb{R}^+}^p=\|\mathcal{H}\phi\|_{p,\mathbb{R}^+}^p
\ge\sum_{m=1}^N\|w\Phi\|_{p,I_m}^p\\=\sum_{m=1}^N|\lambda_m|^p\|wF_m\|_{p,I_m}^p
\ge\frac{1}{2^p}\sum_{m=1}^N|\lambda_m|^p\|w(F_m-(F_m)_{I_m})\|_{p,I_m}^p\\>\frac{(\lambda\varepsilon)^p}{2^p}\sum_{m=1}^N|\lambda_m|^p\|f_m\|_{p}^p
\ge\frac{(\lambda\varepsilon)^p}{2^p}\|\phi\|_{p}^p.\end{eqnarray*} Hence, we obtain the estimate
\begin{eqnarray}\label{kok}\|\mathcal{H}\phi-K\phi\|_{p,\mathbb{R}^+}^p>\frac{(\lambda\varepsilon)^p}{2^p}\|\phi\|_{p}^p,\end{eqnarray}
which shows that $a_N(\mathcal{H})\ge\lambda\varepsilon/2$ with $\lambda$ chosen arbitrary close to $1.$ Thus, $a_N(\mathcal{H})\ge\varepsilon/2.$

To consider the situation when $\mathcal{I}_d=\mathcal{I}_2$ we introduce a sequence $\{\xi_k\}_{k=0}^{K}\subset(0,\infty),$ $0<K<\infty,$ as follows. Let $I_{n_1},$ $n_1\in\mathbb{N}_2,$ be the interval of the type $\mathcal{I}_2$ which appears first from the left. Then we put
\begin{equation}\xi_0=c_{n_1},\ \ \ \xi_k=(a^{-1}\circ b)^k(c_{n_1}),\ \ \ k=0,\ldots,K,\end{equation} where $\xi_K\ge c_{7N}.$ Notice that in view of $b(c_n)> a(c_{n+1})$ we have $c_{n_1+1}<\xi_1=a^{-1}(b(c_{n_1})),$ that is $I_{n_1}\subset[\xi_0,\xi_1)=:\Delta_0.$ Moreover, by the same reason all the intervals from $\mathcal{I}_2$ have non-empty intersections with at most two neighbour segments $\Delta_k:=[\xi_k,\xi_{k+1}).$ Now we divide all the intervals of the type $\mathcal{I}_2$ into two subclasses. The first one, $\mathcal{I}_{2,1},$ is consisting of all $I_n\in\mathcal{I}_2,$ $n\in\mathbb{N}_{2,1}\subseteq\mathbb{N}_2,$ which are having non-empty intersections (of measure greater than $0$) with two neighbour segments $[\xi_k,\xi_{k+1}).$ The second subclass $\mathcal{I}_{2,2}$ is consisting of all the rest $I_n\in\mathcal{I}_2,$ $n\in\mathbb{N}_{2,2}\subseteq\mathbb{N}_2,$ which are such that $I_n\subset\Delta_{k(n)}$ for some $k(n).$ Now we determine a dominant subclass $\mathcal{I}_{2,d}$ of intervals $I_n$ from $\mathcal{I}_{2}.$ Such a subclass must be represented by not less than $3N$ intervals $I_n$ of type $\mathcal{I}_{2,1}$ or by not less than $2N$ intervals $I_n$ from $\mathcal{I}_{2,2}.$

Let $\mathcal{I}_{2,d}=\mathcal{I}_{2,1}$ first. As usual we shall take into account exactly $N$ intervals $I_m\in\mathcal{I}_{2,1},$ which have, for instance, serial numbers multiple of 3. Notice that in such a case the corresponding functions $\phi_m$ have disjoint supports. Moreover, for $x\in I_m$ it holds that
\begin{eqnarray*}\Phi(x)=\lambda_m\int_{a(x)}^{a(c_{m+1})}f_m v+
\lambda_m\int_{b(c_m)}^{b(x)}f_m v \\=\lambda_m\int_{a(x)}^{a(c_{m+1})}f_m v\pm\lambda_m\int_{a(c_{m+1})}^{b(c_m)}f_m v+\lambda_m
\int_{b(c_m)}^{b(x)}f_m v\\
=\lambda_mF_m(x)-\lambda_m\int_{a(c_{m+1})}^{b(c_m)}f_m v=:\lambda_m[F_m(x)+\mu_m].\end{eqnarray*}
Thus, similar to the case $\mathcal{I}_d=\mathcal{I}_1$
\begin{eqnarray*}\|\mathcal{H}\phi-K\phi\|_{p,\mathbb{R}^+}^p=\|\mathcal{H}\phi\|_{p,\mathbb{R}^+}^p
\ge\sum_{m=1}^N\|w\Phi\|_{p,I_m}^p\\
=\sum_{m=1}^N|\lambda_m|^p\|w(F_m-\mu_m)\|_{p,I_m}^p
\ge\frac{1}{2^p}\sum_{m=1}^N|\lambda_m|^p\|F_m-(F_m)_{I_m}\|_{p,I_m}^p\\
>\frac{(\lambda\varepsilon)^p}{2^p}\sum_{m=1}^N|\lambda_m|^p\|f_m\|_{p}^p
\ge\frac{(\lambda\varepsilon)^p}{2^p}\sum_{m=1}^N|\lambda_m|^p\|\phi_m\|_{p}^p\ge\frac{(\lambda\varepsilon)^p}{2^p}\|\phi\|_{p}^p,\end{eqnarray*}
and the required estimate $a_N(\mathcal{H})\ge\varepsilon/2$ follows.

If $\mathcal{I}_{2,d}=\mathcal{I}_{2,2}$ then there exist at least $2N$ intervals $I_n$ which are located inside of some segments $\Delta_{k(n)}$. Let us numerate the segments $\Delta_k$ with such $I_n$ inside from $1$ to some $K_0\le K.$ Notice that some of $\Delta_k$ may include more than 1 of the intervals $I_n\in\mathcal{I}_{2,2}.$ Denote $\Delta_{odd}:=\cup_{\textrm{odd}\,k}\Delta_k$ and $\Delta_{even}:=\cup_{\textrm{even}\,k}\Delta_k.$ Since $\mathcal{I}_{2,2}$ is dominating then at least one of the two sets $\Delta_{odd}$ and $\Delta_{even}$ is represented by not less than $N$ intervals $I_m\in\mathcal{I}_{2,2}.$ We shall consider such a set taking into account only $N$ its members $I_m.$
By the construction, for all $x\in I_m$
\begin{eqnarray*}\Phi(x)=\int_{a(x)}^{a(c_{m+1})}\phi v+\int_{a(c_{m+1})}^{b(c_m)}\phi v+
\int_{b(c_m)}^{b(x)}\phi v\\
=\lambda_m\int_{a(x)}^{a(c_{m+1})}f_m v+\int_{a(c_{m+1})}^{b(c_m)}\phi v+
\lambda_m\int_{b(c_m)}^{b(x)}f_m v\pm\lambda_m\int_{a(c_{m+1})}^{b(c_m)}f_m v\\
=\lambda_mF_m(x)+\int_{a(c_{m+1})}^{b(c_m)}\phi v-\lambda_m\int_{a(c_{m+1})}^{b(c_m)}f_m v=:\lambda_m[F_m(x)+\nu_m].\end{eqnarray*}

We have
\begin{eqnarray*}\|\mathcal{H}\phi-K\phi\|_{p,\mathbb{R}^+}^p=\|\mathcal{H}\phi\|_{p,\mathbb{R}^+}^p
\ge\sum_{m=1}^N\|w\Phi\|_{p,I_m}^p\\
=\sum_{m=1}^N|\lambda_m|^p\|w(F_m-\nu_m)\|_{p,I_m}^p
\ge\frac{1}{2^p}\sum_{m=1}^N|\lambda_m|^p\|F_m-(F_m)_{I_m}\|_{p,I_m}^p\\
>\frac{(\lambda\varepsilon)^p}{2^p}\sum_{m=1}^N|\lambda_m|^p\|f_m\|_{p}^p
\ge\frac{(\lambda\varepsilon)^p}{2^p}\sum_{m=1}^N|\lambda_m|^p\|\phi_m\|_{p}^p\ge\frac{(\lambda\varepsilon)^p}{2^p}\|\phi\|_{p}^p.\end{eqnarray*}
This implies the required estimate $a_N(\mathcal{H})\ge \varepsilon/2.$\end{proof}

The next lemma is giving the similar estimate for $a_n(\mathcal{H})$ from above.
\begin{lemma}\label{Key2}
Let $1<p<\infty$ and $0<\varepsilon<\|\mathcal{H}\|.$ Suppose that there exists $N\in\mathbb{N}$ and numbers $0=c_0<c_1\ldots c_{N}<c_{N+1}=\infty$ such that $\mathcal{K}(I_n)=\varepsilon$ for $n=0,\ldots,N-1,$ where $I_n=(c_n,c_{n+1}),$ and $\mathcal{K}(I_{N})\le\varepsilon.$ Then $a_{N+2}(\mathcal{H})\le 7^{1/p}\varepsilon.$
\end{lemma}\begin{proof}
First of all notice that in view of properties of the operator $H$ \begin{equation}\label{kuku'}\mathcal{K}(I_n)=
\sup_{f:\,{\rm supp} f\subseteq [a(c_n),b(c_{n+1})]}\frac{\|w(H-H_{I_n})\|_{q,I_n}}{\|f\|_{p}}.\end{equation}
Moreover, since $\mathcal{K}(I)=\mathcal{{\bar K}}(I):=\sup_{f\not=0}\frac{\|w({\bar H}-{\bar H}_I)\|_{p,I}}{\|f\|_{p}}$
with $${\bar H}(x)=\int_{a(c_n)}^{a(x)}fv+\int_{b(x)}^{b(c_{n+1})}fv,$$ we have
\begin{equation}\label{kuku}\mathcal{K}(I_n)=
\sup_{f:\,{\rm supp} f\subseteq \bigl\{[a(c_n),a(c_{n+1})]\cup [b(c_n),b(c_{n+1})]\bigr\}}\frac{\|w({\bar H}-{\bar H}_I)\|_{q,I_n}}{\|f\|_{p}}.\end{equation}

Let $f\in L^p$ be such that $\|f\|_{p,\mathbb{R}^+}=1.$ We define $$Kf=\sum_{n=0}^N K_{I_n}f,\hspace{1cm}K_{I}f(x)=\chi_{I}(x)w(x)H_I.$$
Then $K$ is bounded on $L^p$ and has rank not grater than $N+1.$ We write
\begin{eqnarray*}
\|\mathcal{H}f-Kf\|_{p}^p=\sum_{n=0}^N\|\mathcal{H}f-K_{I_n}f\|_{p, I_n}^p=
\sum_{n=0}^N\|w(H-H_{I_n})\|_{p, I_n}^p.\end{eqnarray*}
Now divide all the intervals $I_n,$ $0\le n\le N,$ into three groups $\mathcal{I}_1,$ $\mathcal{I}_{2,1}$ and
$\mathcal{I}_{2,2}$ as in the proof of Lemma \ref{Key}. Namely, let
$$\mathcal{I}_1=\biggl\{I_n,\,n\in\mathbb{N}_1\subseteq\bigl\{0,\ldots,N\bigr\}\colon b(c_n)\le a(c_{n+1})\biggr\},$$
\begin{eqnarray*}\mathcal{I}_{2,1}=\biggl\{I_n,\,n\in\mathbb{N}_{2,1}\subseteq\bigl\{0,\ldots,N\bigr\}\colon b(c_n)> a(c_{n+1})\,{\rm and}\\
{\rm there}\,{\rm exists}\,{\rm two}\,{\rm neighbour}\,\Delta_k\,{\rm such}\,{\rm that}\,
{\rm meas}\{I_n\cap\Delta_k\}\not=\emptyset\biggr\},
\end{eqnarray*}
\begin{eqnarray*}\mathcal{I}_{2,2}=\biggl\{I_n,\,n\in\mathbb{N}_{2,2}\subseteq\bigl\{0,\ldots,N\bigr\}\colon b(c_n)> a(c_{n+1})\,{\rm and}\\
{\rm there}\,{\rm exists}\,{\rm the}\,{\rm only}\,\Delta_k\,{\rm such}\,{\rm that}\,
{\rm meas}\{I_n\cap\Delta_k\}\not=\emptyset\biggr\}.\end{eqnarray*} Notice that $\mathbb{N}_{1}\cap\mathbb{N}_{2,1}\cap\mathbb{N}_{2,2}=\emptyset.$ First consider the class $\mathcal{I}_1$. Assume that $n\in\mathbb{N}_1$ counts only odd (or only even) serial numbers from $\mathbb{N}_1.$
Then, in view of \eqref{kuku'}, the set consists of those $I_n$ whose norms $\mathcal{K}(I_n)$ are achieved on functions with non-overlapping supports. Thus, by \eqref{kuku'} \begin{eqnarray*}\sum_{n\in\mathbb{N}_1}\|w(H-H_{I_n})\|_{p, I_n}^p=
\sum_{{\rm odd}\,n\in\mathbb{N}_1}\|w(H-H_{I_n})\|_{p, I_n}^p\\+\sum_{{\rm even}\,n\in\mathbb{N}_1}\|w(H-H_{I_n})\|_{p, I_n}^p
\le2\sum_{n\in\mathbb{N}_1}\mathcal{K}^p(I_n)\|f\|_{p}^p\\\le2\varepsilon^p\sum_{n\in\mathbb{N}_1}\|f\|_{p}^p
\le2\varepsilon^p\|f\|_{p,\mathbb{R}^+}^p=\gamma_1\varepsilon^p,\hspace{5mm}\gamma_1=2.\end{eqnarray*}

Properties of the class $\mathcal{I}_{2,1}$ (see proof of Lemma \ref{Key}) and \eqref{kuku'} allow us to obtain similar estimate for $I_n$ with $n\in\mathbb{N}_{2,1}:$
 \begin{eqnarray*}\sum_{n\in\mathbb{N}_{2,1}}\|w(H-H_{I_n})\|_{p, I_n}^p\le\gamma_2\varepsilon^p,\hspace{5mm}\gamma_2=3.\end{eqnarray*}
Finally, taking into account \eqref{kuku} we can give similar estimate for all $I_n$ from the class $\mathcal{I}_{2,2}:$
 \begin{eqnarray*}\sum_{n\in\mathbb{N}_{2,2}}\|w(H-H_{I_n})\|_{p, I_n}^p\le\gamma_3\varepsilon^p,\hspace{5mm}\gamma_3=2.\end{eqnarray*}
In summary, we have \begin{eqnarray*}
\|\mathcal{H}f-Kf\|_{p}^p\le \sum_{i=1,2,3}\gamma_i\varepsilon^p=7\varepsilon^p,\end{eqnarray*} which yields the required upper estimate
$a_{N+2}(\mathcal{H})\le 7^{1/p}\varepsilon.$\end{proof}

\subsection{Denotations and technical lemmas}

Given boundaries $a(x),b(x)$ put a sequence $\{\xi_k\}_{k\in\mathbb{Z}}\subset(0,\infty)$ such that
\begin{equation}\label{se}\xi_0=1,\ \ \ \xi_k=(a^{-1}\circ b)^k(1),\ \ \ k\in\mathbb{Z},\end{equation} and denote
\begin{eqnarray*}
\nu_{k}:=\sup_{t\in(\xi_k,\xi_{k+1})}\biggl(\int_{b^{-1}(\sigma(t))}^{a^{-1}(\sigma(t))}w^p(x) \mathrm{d}x\biggr)^{\frac{1}{p}} \biggl(\int_{a(t)}^{b(t)}v^{p'}(y)\mathrm{d}y\biggr)^{\frac{1}{p'}},\\
{\bar \nu}_{k}:=\sup_{t\in(\xi_k,\xi_{k+1})}\biggl(\int_{b^{-1}(\sigma(t))}^{a^{-1}(\sigma(t))} w^p(x)\chi_{[\xi_k,\xi_{k+1}]}(x)\mathrm{d}x\biggr)^{\frac{1}{p}} \biggl(\int_{a(t)}^{b(t)}v^{p'}(y)\mathrm{d}y\biggr)^{\frac{1}{p'}},\\
{\tilde \nu}_{k}:=\biggl(\int_{\xi_k}^{\xi_{k+1}}w^p(x)\mathrm{d}x\biggr)^{\frac{1}{p}} \biggl(\int_{a(\sigma^{-1}(b(\xi_k)))}^{b(\sigma^{-1}(a(\xi_{k+1})))}v^{p'}(y) \mathrm{d}y\biggr)^{\frac{1}{p'}},\end{eqnarray*} where $\sigma(t)$ is a fairway-function satisfying $a(t)<\sigma(t)<b(t),$ $t>0$ and
\begin{equation}\label{fairway}
\int_{a(t)}^{\sigma(t)}v^{p'}(y)\mathrm{d}y=\int_{\sigma(t)}^{b(t)}v^{p'}(y)\mathrm{d}y,\ \ \ t>0,
\end{equation} (see \cite[Definition 2.4]{SU} for details).
Further, for any $k\in\mathbb{Z}$ we introduce a sequence $\{x_j\},$ $-j_a(k)\le j(k)\le j_b(k)-1,$ analogous to that in \cite[Lemmas 2.7, 2.8]{SU}:
\begin{enumerate}\label{sequ}
\item $x_{-j_a(k)}=\xi_k,$ $x_0=\sigma^{-1}(b(\xi_k))=\sigma^{-1}(a(\xi_{k+1})),$ $x_{j_b(k)}=\xi_{k+1};$
\item if $\sigma^{-1}(a(x_0))\le \xi_k$ then $j_a(k)=1;$
\item if $\sigma^{-1}(b(x_0))\ge \xi_{k+1}$ then $j_b(k)=1;$
\item if $\sigma^{-1}(a(x_0))> \xi_k$ then $j_a(k)>1$ and $x_{j(k)-1}=\sigma^{-1}(a(x_{j(k)}))$ for $x_{j(k)-1}>\xi_k,$ $j_a(k)+2\le j(k)\le 0;$
\item if $\sigma^{-1}(b(x_0))> \xi_{k+1}$ then $j_b(k)>1$ and $x_{j(k)+1}=\sigma^{-1}(b(x_{j(k)}))$ for $x_{j(k)+1}<\xi_{k+1},$ $0\le j(k)\le j_b(k)-2.$
\end{enumerate}

Further, define $m=\sum_{l<k}\sum_{i=-j_a}^{j_b-1}i(l)+\sum_{-j_a\le i\le j(k)<j_b}i(k)$ and put
\begin{eqnarray*}
\mu_{m}:=\biggl(\int_{x_m}^{x_{m+1}}w^p(x)\mathrm{d}x\biggr)^{\frac{1}{p}} \biggl(\int_{a(x_m)}^{b(x_{m+1})}v^{p'}(y)\mathrm{d}y \biggr)^{\frac{1}{p'}}\\
=\biggl(\int_{x_{j(k)}}^{x_{j(k)+1}}w^p(x)\mathrm{d}x\biggr)^{\frac{1}{p}} \biggl(\int_{a(x_{j(k)})}^{b(x_{j(k)+1})}v^{p'}(y)\mathrm{d}y
\biggr)^{\frac{1}{p'}}.\end{eqnarray*}

It follows from \cite[Theorem 4.1]{SU} that \begin{eqnarray}\label{HK}\beta_p{\tilde \nu}_{k}\le \beta_p{\bar \nu}_{k}\le\beta_p\sup_{k\in\mathbb{Z}}\nu_k\le\|\mathcal{H}\|_{L^p\to L^p}\le \gamma_p\sup_{m\in\mathbb{Z}}\mu_m.\end{eqnarray}
We establish our result in terms of sequences $\mu_m$ and $\nu_k.$ Meanwhile, in view of ${\tilde \nu}_{k}\le{\bar \nu}_{k}\le\nu_k$ two other sequences ${\tilde \nu}_{k}$ and ${\bar \nu}_{k}$ can be taken instead of $\nu_k$ for checking necessary condition for $\mathcal{H}$ to be in $\mathbf{S}_\alpha,$ $\alpha>0.$ Notice that the sequence ${\tilde \nu}_{k}$ is consisting of the smallest elements and being the most convenient for calculation.

We shall need the three statements below to prove our lower estimate.
\begin{lemma}\label{prev}
Let $c\in I$ be chosen so that $W_{[d,c]}=W_I/2.$ Suppose that $0<\varepsilon<\|\mathcal{H}\|$ and assume that $\sharp S_I(4\varepsilon/\beta_p)\ge 4,$ where
$$S_I(\varepsilon):=\{k\in\mathbb{Z}\colon {\bar \Delta}_k\subset I,\ \nu_{k}>\varepsilon\}.$$ Then $\mathcal{K}(I)>\varepsilon.$
\end{lemma} \begin{proof}
Since $\sharp S_I(4\varepsilon/\beta_p)\ge 4$ then at least one of the intervals $I_d:=[d, b^{-1}(a(c))],$ $I_e:=[a^{-1}(b(c)),e],$ say $I_d,$ contains not less than one member of $S_I(4\varepsilon/\beta_p)$ not hidden by "viewless zone" $[b^{-1}(a(c)),a^{-1}(b(c))].$ Therefore, by Lemma \ref{lem1} and \eqref{HK} \begin{eqnarray*}\mathcal{K}(I)\ge \frac{1}{4}\biggl[
\sup_{f\colon {\rm supp}\,f\subseteq[a(d),a(c)]}\frac{\|w_dH\|_{p}}{\|f\|_p}+
\sup_{f\colon {\rm supp}\,f\subseteq[b(c),b(d)]}\frac{\|w_eH\|_{p}}{\|f\|_p}\biggr]\\
\ge \frac{\beta_p}{4}\nu_{k}>\varepsilon.
\end{eqnarray*}
\end{proof}

\begin{lemma}\label{prev'}
Let $\varepsilon >0$ and $N=N(\varepsilon)$ be the length of the sequence $\{c_n\}_{n=0}^{7N+1}$ from Lemma \ref{Key} with $c_0=0$ and $c_{7N+1}=\infty.$ Then
$$\sharp\{k\in\mathbb{Z}\colon \nu_k>4\varepsilon/\beta_p\}\le 28N(\varepsilon).$$ \end{lemma}\begin{proof}
We have $$\sharp\{k\in\mathbb{Z}\colon c_n\in{\bar \Delta}_k\ \textrm{for}\ \textrm{some}\ n,\,1\le n\le 7N\}\le 7N.$$
For every $k\in\mathbb{Z},$ which is not included in the above set, ${\bar \Delta}_k\subset I_n$ for some $n,$ $0\le n\le 7N-1.$ Then by Lemma \ref{prev}
$$\sharp\{k\in\mathbb{Z}\colon \nu_k>4\varepsilon/\beta_p\}\le 3.$$ Thus,
\begin{eqnarray*}\sharp\{k\in\mathbb{Z}\colon \nu_k>4\varepsilon/\beta_p\}\le
\sum_{k=0}^{7N-1}\sharp\{k\in\mathbb{Z}\colon {\bar \Delta}_k\subseteq I_n,\,\nu_k>4\varepsilon/\beta_p\}\\+7N\le
3\cdot 7N+7N=28N.\end{eqnarray*}
\end{proof}

\begin{lemma}\label{KK} We have for all $t>0:$
$$\sharp\{k\in\mathbb{Z}\colon \nu_{k}>t\}\le 28 \sharp\{n\in\mathbb{N}\colon a_n(\mathcal{H})\ge \beta_pt/8\}.$$
\end{lemma}\begin{proof} By Lemma \ref{Key}
$$\sharp\{n\in\mathbb{N}\colon a_n(\mathcal{H})\ge 2^{-1}\varepsilon\}\ge N(\varepsilon).$$ Therefore, by Lemma \ref{prev'} 
\begin{eqnarray*}
\sharp\{k\in\mathbb{Z}\colon \nu_k>t\}\le 28 N(\beta_pt/4)
\le 28\sharp\{n\in\mathbb{N}\colon a_n(\mathcal{H})\ge \beta_pt/8\}.\end{eqnarray*}
\end{proof}

The next lemma will help us to estimate $\|\mathcal{H}\|_{\mathbf{S}_\alpha}$ from above.

\begin{lemma}\label{KKK} Let $\delta_m=(x_m,x_{m+1}),$ $I_n=(c_n,c_{n+1})$ and $x_m<c_1<c_2<\ldots<c_{l-1}<x_{m+1},$ $l>1.$ Then
$$\sum_{i=1}^{l-1}\biggl(\int_{x_m}^{x_{m+1}}\!\!\!w^p\chi_{I_i}\biggr)^{\frac{1}{p}} \biggl(\int_{a(x_m)}^{b(x_{m+1})}\!\!\!v^{p'}\bigl[\chi_{[a(c_i),a(c_{i+1})]}+\chi_{[b(c_i),b(c_{i+1})]}\bigr]\biggr)^{\frac{1}{p'}} \!\!\!\le\mu_m .$$
\end{lemma}\begin{proof} The statement is following by H\"{o}lder's inequality.
\end{proof}

\subsection{The lower and upper estimates for the Schatten norms} \begin{theorem}\label{main'} For all $\alpha>0,$
$$\sum_{k\in\mathbb{Z}}\nu_k^\alpha\le28\biggl(\frac{8}{\beta_p}\biggr)^\alpha \sum_{n\in\mathbb{N}}a_n^\alpha(\mathcal{H}).$$
\end{theorem}\begin{proof} By \cite[Proposition II.1.8]{BenShar} and in view of Lemma \ref{KK}
\begin{eqnarray*}\sum_{k\in\mathbb{Z}}\nu_k^\alpha=\alpha\int_0^\infty t^{\alpha-1}\sharp\{k\in\mathbb{Z}\colon \nu_k>t\}\mathrm{d}t\\
\le 28\alpha\int_0^\infty t^{\alpha-1}\sharp\{n\in\mathbb{N}\colon a_n(\mathcal{H})\ge \beta_pt/8\}\mathrm{d}t\\=28\biggl(\frac{8}{\beta_p}\biggr)^\alpha
\int_0^\infty \tau^{\alpha-1}\sharp\{n\in\mathbb{N}\colon a_n(\mathcal{H})\ge \tau\}\mathrm{d}\tau\\=
28\biggl(\frac{8}{\beta_p}\biggr)^\alpha\sum_{n\in\mathbb{N}}a_n^\alpha(\mathcal{H}).
\end{eqnarray*}
\end{proof}

\begin{theorem}\label{notmain} Let $\alpha>1.$ Then
$$\sum_{n\in\mathbb{N}}a_n^\alpha(\mathcal{H})\le 2^\alpha\zeta(\alpha)7^{\alpha/p}\gamma_p^\alpha\sum_{m\in\mathbb{Z}}\mu_{m}^\alpha.$$
\end{theorem}\begin{proof} Given $N=N(\varepsilon)$ there exist two options for positioning of the intervals $I_n,$ $0\le n\le N-1$ with respect to the segments $\delta_m:=[x_m,x_{m+1}):$

(1) two neighbour points, say $c_{n_0}$ and $c_{n_0+1},$ are in different intervals $\delta_m,$ say $\delta_{m(n_0)}$ and $\delta_{m(n_0)+1},$ where $m(n_0)<m(n_0+1);$

(2) not less than two neighbour points, say $c_{n_1},\ldots,c_{n_1+l-1}$ with $l>1,$ are in the same interval $\delta_m,$ that is $m(n_1)=m(n_1+1)=\ldots=m(n_1+l-1)$ (or $I_i\subseteq \delta_{m(n_1)}$ for $n_1\le i\le n_1+l-1$), where $l>1$.

By Lemma \ref{lem1} and \eqref{HK} we have
\begin{eqnarray*}
\varepsilon=\mathcal{K}(I_n)\le 2\gamma_p\sup_{m(n_0)\le m\le m(n_0+1)}\mu_m=:2\,\gamma_p\,\mu_{m_0}(n)
\end{eqnarray*} in the first situation. In the second case, by Lemmas \ref{lem1} and \ref{KKK}
\begin{eqnarray*}
\varepsilon l=\sum_{i=n_1}^{n_1+l-1}\mathcal{K}(I_i)\le 2\,\gamma_p\,\mu_{m_1}(n).
\end{eqnarray*} We have
\begin{eqnarray*}N(\varepsilon)=\sharp\bigl\{n\in\mathbb{N}\colon \mu_{m_0}(n)\ge \frac{\varepsilon}{2\gamma_p}\bigr\}+
\sum_{l>1}\sharp\bigl\{n\in\mathbb{N}\colon \mu_{m_1}(n)\ge \frac{\varepsilon l}{2\gamma_p}\bigr\}\\
\le\sum_{l=1}^\infty\sharp\bigl\{n\in\mathbb{N}\colon \mu_{m}(n)\ge \frac{\varepsilon l}{2\gamma_p}\bigr\}
\le\sum_{n=1}^\infty\sharp\bigl\{m\in\mathbb{Z}\colon \mu_{m}\ge \frac{\varepsilon n}{2\gamma_p}\bigr\}.
\end{eqnarray*} On the strength of Lemma \ref{Key2} $$\sharp\{n\in\mathbb{N}\colon a_n(\mathcal{H})> 7^{1/p}\varepsilon\}\le N(\varepsilon)+1\le 2N(\varepsilon).$$
Thus, by \cite[Proposition II.1.8]{BenShar}
\begin{eqnarray*}\sum_{n\in\mathbb{N}}a_n^\alpha(\mathcal{H})=\alpha\int_0^\infty t^{\alpha-1}\sharp\{n\in\mathbb{N}\colon a_n(\mathcal{H})>t\}\mathrm{d}t\\
\le 2\alpha\int_0^\infty t^{\alpha-1}N(t/7^{1/p})\mathrm{d}t\\\le 2\alpha\int_0^\infty t^{\alpha-1}
\sum_{n=1}^\infty\sharp\bigl\{m\in\mathbb{Z}\colon \mu_{m}\ge \frac{t n}{2\cdot 7^{1/p}\gamma_p}\bigr\}\mathrm{d}t\\
\le 2^\alpha \alpha 7^{\alpha/p}\gamma_p^\alpha\int_0^\infty \tau^{\alpha-1}
\sum_{n=1}^\infty \frac{1}{n^{\alpha}}\sharp\{m\in\mathbb{Z}\colon \mu_{m}\ge \tau\}\mathrm{d}\tau\\=
2^\alpha\zeta(\alpha)7^{\alpha/p}\gamma_p^\alpha\sum_{m\in\mathbb{Z}}\mu_{m}^\alpha.
\end{eqnarray*}
\end{proof}

It follows from Theorem \ref{notmain} and properties of the intervals $\delta_m$ that
\begin{equation}\label{int}
\|\mathcal{H}\|_{\mathbf{S}_\alpha}\ll\biggl(\sum_{m\in\mathbb{Z}}\mu_{m}^\alpha\biggr)^{\frac{1}{\alpha}}\ll \mathcal{V}\end{equation} for all $1<\alpha,p<\infty,$ where
\begin{equation}\label{a<p}
\mathcal{V}^\alpha:=
\int_0^\infty \biggl[\int_{b^{-1}(t)}^{a^{-1}(t)}w^p\biggr]^{\frac{\alpha}{p}}
\biggl[\int_{a(\sigma^{-1}(t))}^{b(\sigma^{-1}(t))}v^{p'}\biggr] ^{\frac{\alpha}{p'}-1}v^{p'}(t)\mathrm{d}t<\infty.
\end{equation} By the same reason we have
\begin{equation}\label{int'}
\|\mathcal{H}\|_{\mathbf{S}_\alpha}\ll\biggl(\sum_{m\in\mathbb{Z}}\mu_{m}^\alpha\biggr)^{\frac{1}{\alpha}}\ll \mathcal{W},\end{equation} for $\alpha\ge p,$ where
\begin{equation}\label{a>p}\mathcal{W}^\alpha:=
\int_0^\infty \biggl[\int_{b^{-1}(\sigma(t))}^{a^{-1}(\sigma(t))}w^p\biggr]^{\frac{\alpha}{p}-1}
\biggl[\int_{a(t)}^{b(t)}v^{p'}\biggr]^{\frac{\alpha}{p'}}w^{p}(t)\mathrm{d}t<\infty.
\end{equation}

For establishing \eqref{int} notice that $$\sum_{m\in\mathbb{Z}}\mu_{m}^\alpha=\sum_{k\in\mathbb{Z}}
\sum_{-j_a(k)\le j\le j_b(k)-1}\mu_{j(k)}^\alpha=:\sum_{k\in\mathbb{Z}}J_k.$$ Besides, by \cite[Lemmas 2.7 and 2.8]{SU}  it holds that \begin{equation}\label{U}\int_{a(x_{j(k)})}^{b(x_{j(k)+1})}v^{p'}\approx \int_{a(t)}^{b(t)}v^{p'},\hspace{1cm}t\in[x_{j(k)},x_{j(k)+1}].\end{equation} The rest follows from arguments analogous to those in the proof of Theorem \ref{lom} (see \eqref{U1} -- \eqref{U10}). The estimate \eqref{int'} is going on the strength of \eqref{U} and by the following facts: \begin{eqnarray*}J_k
\simeq
\sum_{-j_a(k)\le j\le -1}\int_{x_{j(k)}}^{x_{j(k)+1}}\biggl(\int_{x_{j(k)}}^{t}w^p\biggr)^{\frac{\alpha}{p}-1}w^p(t)\mathrm{d}t
\biggl(\int_{a(x_{j(k)})}^{b(x_{j(k)+1})}v^{p'}\biggr)^{\frac{\alpha}{p'}}\\
+
\sum_{0\le j\le j_b(k)-1}\int_{x_{j(k)}}^{x_{j(k)+1}}\biggl(\int_t^{x_{j(k)+1}}w^p\biggr)^{\frac{\alpha}{p}-1}w^p(t)\mathrm{d}t
\biggl(\int_{a(x_{j(k)})}^{b(x_{j(k)+1})}v^{p'}\biggr)^{\frac{\alpha}{p'}}\\
\le
\sum_{-j_a(k)\le j\le -1}\int_{x_{j(k)}}^{x_{j(k)+1}}\biggl(\int_{b^{-1}(\sigma(t))}^{t}w^p\biggr)^{\frac{\alpha}{p}-1}w^p(t)\mathrm{d}t
\biggl(\int_{a(x_{j(k)})}^{b(x_{j(k)+1})}v^{p'}\biggr)^{\frac{\alpha}{p'}}\\
+
\sum_{0\le j\le j_b(k)-1}\int_{x_{j(k)}}^{x_{j(k)+1}}\biggl(\int_t^{a^{-1}(\sigma(t))}w^p\biggr)^{\frac{\alpha}{p}-1}w^p(t)\mathrm{d}t
\biggl(\int_{a(x_{j(k)})}^{b(x_{j(k)+1})}v^{p'}\biggr)^{\frac{\alpha}{p'}}.
\end{eqnarray*}

In some special cases of the boundaries $a,b$ and weights $v,w$ the two-sided estimate obtained in Theorems \ref{main'} and \ref{notmain} for the Schatten-von Neumann norm $\|\mathcal{H}\|_{\mathbf{S}_\alpha}$ of the Hardy-Steklov operator $\mathcal{H}$ becomes symmetric giving a criterion for $\mathcal{H}$ to be in a class $\mathbf{S}_\alpha$ for $1<\alpha<\infty.$

\begin{example}\label{ex}
Let $a(x)=Ax,$ $b(x)=Bx$ for some $0<A<B<\infty$ and $v\equiv 1.$ Then, by definition \eqref{fairway}, the fairway $\sigma(x)=(A+B)x/2.$

Given $w$, $\alpha$ and $p$ consider a quantity $\Lambda_k^\alpha:={\tilde \nu}_{k-1}^\alpha+{\tilde \nu}_{k}^\alpha+{\tilde \nu}_{k+1}^\alpha.$
Denote $\xi_k^+:=\sigma^{-1}(b(\xi_k))$ and $\xi_k^-:=\sigma^{-1}(a(\xi_k)).$ Then
\begin{eqnarray*}\Lambda_k^p\approx\bigl[b(\xi_k^-)-a(\xi_k^-)\bigr]^{p-1}\int_{\xi_{k-1}}^{\xi_{k}}w^p
+ \bigl[b(\xi_k^+)-a(\xi_k^+)\bigr]^{p-1}\int_{\xi_{k}}^{\xi_{k+1}}w^p\\
+ \bigl[b(\xi_{k+1}^+)-a(\xi_{k+1}^+)\bigr]^{p-1}\int_{\xi_{k+1}}^{\xi_{k+2}}w^p=:{\tilde \Lambda}^p_k\end{eqnarray*}
provided $\xi_k^+=\xi_{k+1}^-.$ By the construction $\xi_k^-=2A\xi_k/(A+B),$ $\xi_k^+=2B\xi_k/(A+B)$ and, therefore, $a(\xi_k^+)=2AB\xi_k/(a+B)=b(\xi_k^-).$ Since $b(\xi_k)-a(\xi_k)>b(\xi_k^-)-\sigma(\xi_k^-)$ and $b(\xi_k)-a(\xi_k)>\sigma(\xi_k^+)-a(\xi_k^+)$ then there exist $1<c_1,c_2<\infty$ such that $B-A=A c_1[2B/(A+B)-1]$ and $B-A=B c_2[1-2A/(A+B)],$ namely, $c_1=1+B/A$ and $c_2=1+A/B.$ Together with property \eqref{fairway} it gives:
\begin{eqnarray*}
b(\xi_k^-)-a(\xi_k^-)=2[b(\xi_k^-)-\sigma(\xi_k^-)]=2[b(\xi_k)-a(\xi_k)]/c_1,\\
b(\xi_k^+)-a(\xi_k^+)=2[\sigma(\xi_k^+)-a(\xi_k^+)]=2[b(\xi_k)-a(\xi_k)]/c_2.
\end{eqnarray*} Analogously,
\begin{eqnarray*}
b(\xi_{k+1}^+)-a(\xi_{k+1}^+)=2[b(\xi_{k+1})-a(\xi_{k+1})]/c_2.
\end{eqnarray*} Moreover, since $\xi_{k+1}=\frac{B}{A}\xi_k$ then $$b(\xi_{k+1})-a(\xi_{k+1})=[B-A]\xi_{k+1}=\frac{B}{A}[B-A]\xi_k=\frac{B}{A}[b(\xi_{k})-a(\xi_{k})]$$ and $b(\xi_{k+1})-a(\xi_{k})=c_1[b(\xi_{k})-a(\xi_{k})].$
Therefore, in view of $1/c_1<1/c_2,$
\begin{eqnarray*}
{\tilde \Lambda}^p_k=\frac{2^{p-1}}{c_1^{p-1}}[b(\xi_k)-a(\xi_k)]^{p-1}\int_{\xi_{k-1}}^{\xi_{k}}w^p+
\frac{2^{p-1}}{c_2^{p-1}}[b(\xi_k)-a(\xi_k)]^{p-1}\int_{\xi_{k}}^{\xi_{k+1}}w^p\\
+\frac{2^{p-1}}{c_2^{p-1}}[b(\xi_{k+1})-a(\xi_{k+1})]^{p-1}\int_{\xi_{k+1}}^{\xi_{k+2}}w^p\\
>\frac{2^{p-1}}{c_1^{p-1}}[b(\xi_k)-a(\xi_k)]^{p-1}\int_{\xi_{k-1}}^{\xi_{k+1}}w^p+
\frac{2^{p-1}}{c_2^{p-1}}\frac{B^{p-1}}{A^{p-1}}[b(\xi_{k})-a(\xi_{k})]^{p-1}\int_{\xi_{k+1}}^{\xi_{k+2}}w^p\\
>\frac{2^{p-1}}{c_1^{p-1}}[b(\xi_k)-a(\xi_k)]^{p-1}\int_{\xi_{k-1}}^{\xi_{k+2}}w^p=
\frac{2^{p-1}}{c_1^{2(p-1)}}[b(\xi_{k+1})-a(\xi_k)]^{p-1}\int_{\xi_{k-1}}^{\xi_{k+2}}w^p.
\end{eqnarray*} Further, by Theorem \ref{main'} and in view of ${\tilde \nu}_k\le\nu_k$ we obtain for $\alpha\ge p$
\begin{eqnarray*}\|\mathcal{H}\|_{\mathbf{S}_\alpha}^\alpha&\gg&\sum_{k\in\mathbb{Z}}{\tilde \nu}_k\ge\frac{1}{3}\sum_{k\in\mathbb{Z}}\Lambda_k^\alpha
\approx\sum_{k\in\mathbb{Z}}{\tilde \Lambda}_k^\alpha\gg\\&\gg&c_1^{-2\alpha/p'}
[b(\xi_{k+1})-a(\xi_k)]^{\frac{\alpha}{p'}}\biggl(\int_{\xi_{k-1}}^{\xi_{k+2}}w^p\biggr)^{\frac{\alpha}{p}}\\
&\gg&c_1^{-2\alpha/p'}
[b(\xi_{k+1})-a(\xi_k)]^{\frac{\alpha}{p'}}\int_{\xi_{k}}^{\xi_{k+1}}\biggl(\int_{\xi_{k-1}}^{t}w^p\biggr)^{\frac{\alpha}{p}-1}w^p(t)\mathrm{d}t\\
&&+c_1^{-2\alpha/p'}
[b(\xi_{k+1})-a(\xi_k)]^{\frac{\alpha}{p'}}\int_{\xi_{k}}^{\xi_{k+1}}\biggl(\int_t^{\xi_{k+2}}w^p\biggr)^{\frac{\alpha}{p}-1}w^p(t)\mathrm{d}t\\
&\gg&
c_1^{-2\alpha/p'}
[b(\xi_{k+1})-a(\xi_k)]^{\frac{\alpha}{p'}}\int_{\xi_{k}}^{\xi_{k+1}}\biggl(\int_{b^{-1}(\sigma(t))}^{a^{-1}(\sigma(t))} w^p\biggr)^{\frac{\alpha}{p}-1}w^p(t)\mathrm{d}t\\
&\ge&
c_1^{-2\alpha/p'}
\int_{\xi_{k}}^{\xi_{k+1}}\biggl(\int_{b^{-1}(\sigma(t))}^{a^{-1}(\sigma(t))} w^p\biggr)^{\frac{\alpha}{p}-1}
[b(t)-a(t)]^{\frac{\alpha}{p'}}w^p(t)\mathrm{d}t.\end{eqnarray*}
\end{example} This means that for this particular case of $a,b$ and $v$ the norm $\|\mathcal{H}\|_{\mathbf{S}_\alpha},$ $1<\alpha<\infty,$ is equivalent to the functional $$F:=\biggl(\int_{\xi_{k}}^{\xi_{k+1}}\biggl(\int_{b^{-1}(\sigma(t))}^{a^{-1}(\sigma(t))} w^p\biggr)^{\frac{\alpha}{p}-1}
[b(t)-a(t)]^{\frac{\alpha}{p'}}w^p(t)\mathrm{d}t\biggr)^{\frac{1}{\alpha}},$$ which finiteness is necessary and sufficient for belonging $\mathcal{H}$ to the Schatten-von Neumann classes $\mathbf{S}_{\alpha}$ for all $1<\alpha<\infty.$

\section{An alternative upper estimate}

\begin{theorem}\label{lom} Let $1<\alpha,p<\infty.$
Suppose that the operator $\mathcal{H}:L^p(\mathbb{R}^+)\to L^p(\mathbb{R}^+)$ is compact.
If $\alpha\le p$ then
\begin{equation}\label{p<1}
\|\mathcal{H}\|_{\mathbf{S}_\alpha}\ll \mathcal{V}.
\end{equation} In the case $p\le\alpha$ we have
\begin{equation}\label{p>1}
\|\mathcal{H}\|_{\mathbf{S}_\alpha}\ll \mathcal{W}.\end{equation}
\end{theorem}

\begin{proof} Given boundaries $a(x),b(x)$ put the sequence $\{\xi_k\}_{k\in\mathbb{Z}}\subset(0,\infty)$ defined by the formula \eqref{se} and split the operator $\mathcal{H}$ into the sum $$\mathcal{H}f(x)=\sum_{k\in\mathbb{Z}} \bigl[T_{k,1}f(x)+T_{k,2}f(x)+S_{k,1}f(x)+S_{k,2}f(x)\bigr]$$ of four block-diagonal operators
$$T_if(x)=\sum_{k\in\mathbb{Z}}T_{i,k}f(x),\ \ \ S_if(x)=\sum_{k\in\mathbb{Z}}S_{i,k}f(x),\ \ \ i=1,2,$$ of the forms
\begin{eqnarray*}T_{k,1}f(x):=w(x)\int_{a(x)}^{a(\xi_k)}f(y)v(y)\mathrm{d}y, \hspace{.2cm}x\in[\sigma^{-1}(a(\xi_{k+1})),\xi_{k+1}],\\
T_{k,2}f(x)=w(x)\int_{a(x)}^{b(\xi_k)}f(y)v(y)\mathrm{d}y, \hspace{.2cm}x\in[\xi_k,\sigma^{-1}(b(\xi_{k}))],\\
S_{k,1}f(x)=w(x)\int_{b(\xi_k)}^{b(x)}f(y)v(y)\mathrm{d}y, \hspace{.2cm}x\in[\xi_k,\sigma^{-1}(b(\xi_{k}))],\\
S_{k,2}f(x)=w(x)\int_{a(\xi_k)}^{b(x)}f(y)v(y)\mathrm{d}y ,\hspace{.2cm}x\in[\sigma^{-1}(a(\xi_{k+1})),\xi_{k+1}].\end{eqnarray*}
By \cite[Lemma 2, Th. 2, 3, 4]{LomI} and \cite[Th. 2, 3, 4]{LomII} we obtain
\begin{eqnarray*}
\|\mathcal{H}\|_{\mathbf{S}_\alpha}^\alpha\le\sum_{k}\|T_{k,1}\|_{\mathbf{S}_\alpha}^\alpha+
\sum_{k}\|T_{k,2}\|_{\mathbf{S}_\alpha}^\alpha+\sum_{k}\|S_{k,1}\|_{\mathbf{S}_\alpha}^\alpha
+\sum_{k}\|S_{k,2}\|_{\mathbf{S}_\alpha}^\alpha\\ \ll
\sum_{k}\int_{\sigma^{-1}(a(\xi_{k+1}))}^{\xi_{k+1}}
\biggl(\int_{a(s)}^{a(\xi_{k+1})}v^{p'}\biggr)^{\frac{\alpha}{p'}}
\biggl(\int_{\sigma^{-1}(a(\xi_{k+1}))}^{s}w^p\biggr)^{\frac{\alpha}{p}-1}
w^{p}(s)\mathrm{d}s\\
+\sum_{k}
\int_{\xi_{k}}^{\sigma^{-1}(b(\xi_{k}))}\biggl(\int_{a(s)}^{b(\xi_k)}v^{p'}\biggr)^{\frac{\alpha}{p'}}
\biggl(\int_{\xi_{k}}^{s}w^p\biggr)^{\frac{\alpha}{p}-1}w^{p}(s)\mathrm{d}s\\
+\sum_{k}
\int_{\xi_{k}}^{\sigma^{-1}(b(\xi_{k}))}\biggl(\int_{b(\xi_k)}^{b(s)}v^{p'}\biggr)^{\frac{\alpha}{p'}}
\biggl(\int_s^{\sigma^{-1}(b(\xi_{k}))}w^p\biggr)^{\frac{\alpha}{p}-1}w^{p}(s)\mathrm{d}s\\
+\sum_{k}\int_{\sigma^{-1}(a(\xi_{k+1}))}^{\xi_{k+1}}
\biggl(\int_{a(\xi_{k+1})}^{b(s)}v^{p'}\biggr)^{\frac{\alpha}{p'}}
\biggl(\int_{s}^{\xi_{k+1}}w^p\biggr)^{\frac{\alpha}{p}-1}
w^{p}(s)\mathrm{d}s\\
=:\Sigma_1+\Sigma_2+\Sigma_3+\Sigma_4.
\end{eqnarray*}

Let $\alpha\le p$ first. Any fixed $k\in\mathbb{Z}$ we introduce an additional  sequence $\{x_{j(k)}\}_{j=-j_a(k)}^{j_b(k)}$ analogous to that on the page \pageref{sequ}.

Denote $\xi_k^+:=\sigma^{-1}(b(\xi_{k}))$ and $\xi_{k+1}^-=\sigma^{-1}(a(\xi_{k+1})).$
If $j_b=1$ then $\sigma(\xi_{k+1})\le b(\xi_{k+1}^-)$ and on the strength of the property \eqref{fairway} of $\sigma$
\begin{eqnarray*}
\Sigma_1\le \frac{p}{\alpha}\sum_{k}\biggl(\int_{a(\xi_{k+1})}^{b(\xi_{k+1}^-)}v^{p'}\biggr)^{\frac{\alpha}{p'}}
\biggl(\int_{\xi_{k+1}^-}^{\xi_{k+1}}w^p\biggr)^{\frac{\alpha}{p}}\end{eqnarray*} and
\begin{eqnarray*}
\Sigma_4\le
2^{\alpha/p'} \sum_{k}\int_{\xi_{k+1}^-}^{\xi_{k+1}}
\biggl(\int_{a(s)}^{\sigma(s)}v^{p'}\biggr)^{\frac{\alpha}{p'}}
\biggl(\int_{s}^{\xi_{k+1}}w^p\biggr)^{\frac{\alpha}{p}-1}
w^{p}(s)\mathrm{d}s\\
\le 2^{\alpha/p'}\frac{p}{\alpha}\sum_{k}\biggl(\int_{a(\xi_{k+1}^-)}^{\sigma(\xi_{k+1})}v^{p'}\biggr)^{\frac{\alpha}{p'}}
\biggl(\int_{\xi_{k+1}^-}^{\xi_{k+1}}w^p\biggr)^{\frac{\alpha}{p}}.\end{eqnarray*}
Therefore, by \eqref{fairway} and in view of $\sigma(\xi_{k+1})\le b(\xi_{k+1}^-)$
\begin{eqnarray*}\Sigma_{1,4}:=\Sigma_1+\Sigma_4\ll \sum_{k}\biggl(\int_{a(\xi_{k+1}^-)}^{b(\xi_{k+1}^-)}v^{p'}\biggr)^{\frac{\alpha}{p'}}
\biggl(\int_{\xi_{k+1}^-}^{\xi_{k+1}}w^p\biggr)^{\frac{\alpha}{p}}\\
\le2^{\alpha/p'+1}\sum_{k}\biggl(\int_{a(\xi_{k+1}^-)}^{b(\xi_{k+1}^-)}v^{p'}\biggr)^{-1}
\biggl(\int_{\sigma(\xi_{k+1}^-)}^{b(\xi_{k+1}^-)}v^{p'}\biggr)^{\frac{\alpha}{p'}+1}
\biggl(\int_{\xi_{k+1}^-}^{\xi_{k+1}}w^p\biggr)^{\frac{\alpha}{p}}.\end{eqnarray*}
Notice that if $\sigma(\xi_{k+1}^-)\le t\le b(\xi_{k+1}^-)$ then $b(\xi_{k+1}^-)\le b(\sigma^{-1}(t)),$ $a(\xi_{k+1}^-)\le a(\sigma^{-1}(t))$ and $b^{-1}(t)\le \xi_{k+1}^-<\xi_{k+1}\le a^{-1}(t).$ Besides, by \eqref{fairway} \begin{equation}\label{sigma'}\int_{a(\sigma^{-1}(t))}^{t}v^{p'}(y)\mathrm{d}y =\int_{t}^{b(\sigma^{-1}(t))}v^{p'}(y)\mathrm{d}y,\hspace{1cm}t>0.\end{equation}
Thus,
\begin{eqnarray*}\Sigma_{1,4}\ll \sum_{k}\biggl(\int_{a(\xi_{k+1}^-)}^{b(\xi_{k+1}^-)}\!\!v^{p'}\biggr)^{-1}\!\!\! \int_{\sigma(\xi_{k+1}^-)}^{b(\xi_{k+1}^-)}
\biggl(\int_{t}^{b(\xi_{k+1}^-)}\!\!v^{p'}\biggr)^{\frac{\alpha}{p'}}v^{p'}(t)\mathrm{d}t
\biggl(\int_{\xi_{k+1}^-}^{\xi_{k+1}}\!\!w^p\biggr)^{\frac{\alpha}{p}}\\
\le \sum_{k} \int_{\sigma(\xi_{k+1}^-)}^{b(\xi_{k+1}^-)}
\biggl(\int_{t}^{b(\xi_{k+1}^-)}v^{p'}\biggr)^{\frac{\alpha}{p'}}\biggl(\int_{a(\xi_{k+1}^-)}^{t}v^{p'}\biggr)^{-1}
\biggl(\int_{b^{-1}(t)}^{a^{-1}(t)}w^p\biggr)^{\frac{\alpha}{p}}v^{p'}(t)\mathrm{d}t\\
\le \sum_{k} \int_{\sigma(\xi_{k+1}^-)}^{b(\xi_{k+1}^-)}
\biggl(\int_{t}^{b(\sigma^{-1}(t))}v^{p'}\biggr)^{\frac{\alpha}{p'}}\!\!\biggl(\int_{a(\sigma^{-1}(t))}^{t}v^{p'}\biggr)^{-1}\!
\biggl(\int_{b^{-1}(t)}^{a^{-1}(t)}w^p\biggr)^{\frac{\alpha}{p}}\!\!\!v^{p'}(t)\mathrm{d}t\\
\le 2^{1-\alpha/p'}\sum_{k} \int_{a(\xi_{k+1})}^{b(\xi_{k+1})}
\biggl(\int_{a(\sigma^{-1}(t))}^{b(\sigma^{-1}(t))}v^{p'}\biggr)^{\frac{\alpha}{p'}-1} \biggl(\int_{b^{-1}(t)}^{a^{-1}(t)}w^p\biggr)^{\frac{\alpha}{p}}v^{p'}(t)\mathrm{d}t. \end{eqnarray*}

If $j_b>1$ then we obtain provided $\alpha/p\le 1$:
\begin{eqnarray*}
\Sigma_1&\le& \frac{p}{\alpha}\sum_{k}\biggl(\int_{a(\xi_{k+1}^-)}^{a(\xi_{k+1})}v^{p'}\biggr)^{\frac{\alpha}{p'}}
\biggl(\int_{\xi_{k+1}^-}^{\xi_{k+1}}w^p\biggr)^{\frac{\alpha}{p}}\\&\le&\frac{p}{\alpha}
\sum_{k}\sum_{0\le j\le j_b-1}\biggl(\int_{a(\xi_{k+1})}^{b(\xi_{k+1}^-)}v^{p'}\biggr)^{\frac{\alpha}{p'}}
\biggl(\int_{x_{j(k)}}^{x_{j(k)+1}}w^p\biggr)^{\frac{\alpha}{p}}.
\end{eqnarray*} Further, in view of $\alpha/p-1\le 0$, properties of $\sigma$ and $b(x_{k(j)})\ge\sigma(x_{j(k)+1})$
\begin{eqnarray*}
\Sigma_4=\sum_{k}\sum_{0\le j\le j_b-1}\int_{x_{j(k)}}^{x_{j(k)+1}}
\biggl(\int_{a(\xi_{k+1})}^{b(s)}v^{p'}\biggr)^{\frac{\alpha}{p'}}
\biggl(\int_{s}^{\xi_{k+1}}w^p\biggr)^{\frac{\alpha}{p}-1}
w^{p}(s)\mathrm{d}s\\
\le 2^{\alpha/p'}\sum_{k}\sum_{0\le j\le j_b-1}\int_{x_{j(k)}}^{x_{j(k)+1}}
\biggl(\int_{a(s)}^{\sigma(s)}v^{p'}\biggr)^{\frac{\alpha}{p'}}
\biggl(\int_{s}^{\xi_{k+1}}w^p\biggr)^{\frac{\alpha}{p}-1}
w^{p}(s)\mathrm{d}s\\
\le 2^{\alpha/p'}\frac{p}{\alpha}\sum_{k}\sum_{0\le j\le j_b-1}\biggl(\int_{a(x_{j(k)})}^{\sigma(x_{j(k)+1})} v^{p'}\biggr)^{\frac{\alpha}{p'}}
\biggl(\int_{x_{j(k)}}^{x_{j(k)+1}}w^p\biggr)^{\frac{\alpha}{p}}\\
\le  2^{2\alpha/p'}\frac{p}{\alpha}\sum_{k}\sum_{0\le j\le j_b-1}\biggl(\int_{\sigma(x_{j(k)})}^ {b(x_{j(k)})}v^{p'}\biggr)^{\frac{\alpha}{p'}}
\biggl(\int_{x_{j(k)}}^{x_{j(k)+1}}w^p\biggr)^{\frac{\alpha}{p}}.\end{eqnarray*}
This yields
\begin{eqnarray}\label{U1}\Sigma_{1,4}\ll \sum_{k}\sum_{0\le j\le j_b-1}\biggl(\int_{\sigma(x_{j(k)})}^{b(x_{j(k)})}v^{p'}\biggr)^ {\frac{\alpha}{p'}}
\biggl(\int_{x_{j(k)}}^{x_{j(k)+1}}w^p\biggr)^{\frac{\alpha}{p}}\\
\le 2\sum_{k}\sum_{0\le j\le j_b-1}\biggl(\int_{a(x_{j(k)})}^{b(x_{j(k)})}v^{p'}\biggr)^{-1}
\biggl(\int_{\sigma(x_{j(k)})}^{b(x_{j(k)})}v^{p'}\biggr)^{\frac{\alpha}{p'}+1}
\biggl(\int_{x_{j(k)}}^{x_{j(k)+1}}w^p\biggr)^{\frac{\alpha}{p}}.\nonumber\end{eqnarray}
Here again, if $\sigma(x_{j(k)})\le t\le b(x_{j(k)})$ then $b(x_{j(k)})\le b(\sigma^{-1}(t)),$ $a(x_{j(k)})\le a(\sigma^{-1}(t))$ and $b^{-1}(t)\le x_{j(k)}<x_{j(k)+1}\le a^{-1}(t).$
Therefore, as before
\begin{eqnarray}\label{U2}\Sigma_{1,4}&\ll& \sum_{k}\sum_{0\le j\le j_b-1}\int_{\sigma(x_{j(k)})}^{b(x_{j(k)})}
\biggl(\int_{a(\sigma^{-1}(t))}^{b(\sigma^{-1}(t))}v^{p'}\biggr)^{\frac{\alpha}{p'}-1}\nonumber\\&&\times
\biggl(\int_{b^{-1}(t)}^{a^{-1}(t)}w^p\biggr)^{\frac{\alpha}{p}}v^{p'}(t)\mathrm{d}t. \end{eqnarray}
By the construction the intervals $(\sigma(\xi_{k,j}),b(\xi_{k,j}))$ are disjoint. Thus,
\begin{eqnarray}\label{U3}~~~~~~~~~~\Sigma_{1,4}\ll\sum_{k} \int_{a(\xi_{k+1})}^{b(\xi_{k+1})}
\biggl(\int_{a(\sigma^{-1}(t))}^{b(\sigma^{-1}(t))}v^{p'}\biggr)^{\frac{\alpha}{p'}-1} \biggl(\int_{b^{-1}(t)}^{a^{-1}(t)}w^p\biggr)^{\frac{\alpha}{p}}v^{p'}(t)\mathrm{d}t. \end{eqnarray}
Analogously, one can prove that
\begin{eqnarray}\label{U10}\Sigma_{2,3}:&=&\Sigma_2+\Sigma_3\nonumber\\ &\ll& \sum_{k} \int_{a(\xi_{k})}^{b(\xi_{k})}
\biggl(\int_{a(\sigma^{-1}(t))}^{b(\sigma^{-1}(t))}v^{p'}\biggr)^{\frac{\alpha}{p'}-1} \biggl(\int_{b^{-1}(t)}^{a^{-1}(t)}w^p\biggr)^{\frac{\alpha}{p}}v^{p'}(t)\mathrm{d}t .\end{eqnarray}
Thus, taking into account that $b(\xi_k)=a(\xi_{k+1}),$ we obtain \eqref{p<1}.

Now let $p\le\alpha.$ We have
\begin{eqnarray*}
\Sigma_1\le \frac{p}{\alpha}\sum_{k}\biggl(\int_{a(\xi_{k+1}^-)}^{\sigma(\xi_{k+1}^-)}v^{p'} \biggr)^{\frac{\alpha}{p'}}
\biggl(\int_{\xi_{k+1}^-}^{\xi_{k+1}}w^p\biggr)^{\frac{\alpha}{p}}\\
=\sum_{k}\biggl(\int_{\sigma(\xi_{k+1}^-)}^{b(\xi_{k+1}^-)}v^{p'}\biggr)^{\frac{\alpha}{p'}}
\int_{\xi_{k+1}^-}^{\xi_{k+1}}\biggl(\int_{t}^{\xi_{k+1}}w^p\biggr)^{\frac{\alpha}{p}-1}w^p(t) \mathrm{d}t\\
\le\sum_{k}
\int_{\xi_{k+1}^-}^{\xi_{k+1}}\biggl(\int_{t}^{a^{-1}(\sigma(t))}w^p\biggr)^{\frac{\alpha}{p}-1}
\biggl(\int_{a(t)}^{b(t)}v^{p'}\biggr)^{\frac{\alpha}{p'}}w^p(t)\mathrm{d}t, \end{eqnarray*}
because $\xi_{k+1}\le a^{-1}(\sigma(t))$ and $a(t)\le\sigma(\xi_{k+1}^-)<b(\xi_{k+1}^-)\le b(t)$ for any $t\in[\xi_{k+1}^-,\xi_{k+1}].$
Further,
\begin{eqnarray*}
\Sigma_2\le \sum_{k}
\int_{\xi_{k}}^{\xi_k^+}\biggl(\int_{b^{-1}(\sigma(t))}^tw^p\biggr)^{\frac{\alpha}{p}-1}
\biggl(\int_{a(t)}^{b(t)}v^{p'}\biggr)^{\frac{\alpha}{p'}}w^p(t)\mathrm{d}t\end{eqnarray*} in view of $\xi_k\ge b^{-1}(\sigma(t))$ and $b(\xi_k)\le b(t)$ for any $t\in[\xi_k,\xi_k^+].$

Analogously, we obtain
\begin{eqnarray*}
\Sigma_3\le \sum_{k}
\int_{\xi_{k}}^{\xi_k^+}\biggl(\int_{b^{-1}(\sigma(t))}^tw^p\biggr)^{\frac{\alpha}{p}-1}
\biggl(\int_{a(t)}^{b(t)}v^{p'}\biggr)^{\frac{\alpha}{p'}}w^p(t)\mathrm{d}t\end{eqnarray*} and
\begin{eqnarray*}
\Sigma_4\le \sum_{k}
\int_{\xi_{k+1}^-}^{\xi_{k+1}}\biggl(\int_{t}^{a^{-1}(\sigma(t))}w^p\biggr)^{\frac{\alpha}{p}-1}
\biggl(\int_{a(t)}^{b(t)}v^{p'}\biggr)^{\frac{\alpha}{p'}}w^p(t)\mathrm{d}t.\end{eqnarray*} Thus, we arrive to \eqref{p>1} provided $\alpha/p-1 \ge 0.$ \end{proof}

{\bf Acknowledgements.} The European Commission is grant-giving authority for the research of the author (Project IEF-2009-252784). The work was also partially supported by the Far-Eastern Branch of the Russian Academy of Sciences (Projects 12-I-OMH-01 and 12-II-0-01M-005). The author thanks Dr D.V. Prokhorov from the Computing Centre of the Far Eastern Branch of the Russian Academy of Sciences in Khabarovsk for his interest to the work and useful discussions.

\end{document}